\documentclass[11pt]{amsart}
\usepackage{amssymb,amscd,amsxtra,calc}
\usepackage{cmmib57}
\usepackage{graphicx}
\usepackage{amssymb,amsmath,amsfonts,mathrsfs,enumerate,xcolor}
\usepackage{amsthm}

\usepackage{soul}
\usepackage[citecolor=Navy]{hyperref}
\usepackage[ruled,vlined]{algorithm2e}

\usepackage{lipsum}
\usepackage{amsfonts}
\usepackage{graphicx}
\usepackage{epstopdf}
\usepackage{algorithmic}
\usepackage{float}
\restylefloat{table}
\usepackage{placeins}
\usepackage{array}
\ifpdf
  \DeclareGraphicsExtensions{.eps,.pdf,.png,.jpg}
\else
  \DeclareGraphicsExtensions{.eps}
\fi

\usepackage{subcaption}

\theoremstyle{plain}
    \newtheorem{thm}{Theorem}[section]

    \newtheorem{theorem}[thm]{Theorem}

\theoremstyle{definition}

    \newtheorem{remark}[thm]{Remark}
\theoremstyle{remark}

\newcommand{\authorfootnotes}{\renewcommand\thefootnote{\@fnsymbol\c@footnote}}

\usepackage{caption} 
 \title[Direct New Q-Newton's method Backtracking for m equations in m variables]{A more direct and better variant of New Q-Newton's method Backtracking for m equations in m variables}
 \author{Tuyen Trung Truong}
   \address{Department of Mathematics, University of Oslo, Blindern 0851 Oslo, Norway}
  \email{tuyentt@math.uio.no}
    \date{\today}
    \keywords{Backtracking line search, Convergence guarantee, Newton's method, Rate of convergence, Systems of nonlinear equations}
   \subjclass[2010]{}

\begin{document}
\maketitle
{\centering\footnotesize To the victims in Kongsberg\par}

\begin{abstract} In some (joint) recent papers, the authors have developed a new family of modifications of Newton's method, for which Backtracking line search can be incorporated, for optimization. The new method, called New Q-Newton's method (and its Backtracking version), has good theoretical guarantee (concerning convergence to critical points, avoidance of saddle points and rate of convergence). This method can be used to solve a system of equations $g_1=\ldots =g_N=0$, by applying to the function $f=g_1^2+\ldots +g_N^2$. 

In the special case where the number of equations and the number of variables are the same, Newton's method can also be  used directly to the system, instead of via the function $f$ as above. While there are known problems with the direct application of Newton's method (such as it is known that there are attracting cycles of non-critical points), if it converges then it usually converges fast. Inspired by this fact, in this paper we apply the ideas of New Q-Newton's method directly to such a system, utilising the specialties of the cost function $f=||F||^2$, where $F=(f_1,\ldots ,f_m)$. 

The first algorithm proposed here is a modification of Levenberg-Marquardt algorithm, where we prove some new results on global convergence and avoidance of saddle points. 

The second algorithm proposed here is a modification of New Q-Newton's method Backtracking, where we use the operator $\nabla ^2f(x)+\delta ||F(x)||^{\tau}$ instead of $\nabla ^2f(x)+\delta ||\nabla f(x)||^{\tau}$. This new version is more suitable than New Q-Newton's method Backtracking itself, while currently has better avoidance of saddle points guarantee than Levenberg-Marquardt algorithms. 

Also, a general scheme for second order methods for solving systems of equations is proposed. We will also discuss a way to avoid that the limit of the constructed sequence is a solution of $H(x)^{\intercal}F(x)=0$ but not of $F(x)=0$.

\end{abstract}

\section{The algorithm, the result and proof, and some comments} 

In this short note we define a new variant of New Q-Newton's method Backtracking \cite{truong2021} (developed from \cite{truong-etal}), to better solve systems of equations. If $A$ is a square matrix, we denote by $minsp(A)=\min \{|\lambda |: $ $\lambda$ is an eigenvalue of $A\}$. Also, we denote by $A^{\intercal}$ the transpose of $A$. 

If $A:\mathbb{R}^m\rightarrow \mathbb{R}^m$ is an invertible {\bf symmetric} square matrix, then it is diagonalisable.  Let $V^{+}$ be the vector space generated by eigenvectors of positive eigenvalues of $A$, and $V^{-}$ the vector space generated by eigenvectors of negative eigenvalues of $A$. Then $pr_{A,+}$ is the orthogonal projection from $\mathbb{R}^m$ to $V^+$, and  $pr_{A,-}$ is the orthogonal projection from $\mathbb{R}^m$ to $V^-$. As usual, $Id$ means the $m\times m$ identity matrix.  Given $F:\mathbb{R}^m\rightarrow \mathbb{R}^m$ a $C^2$ function, we denote by $H(x)=JF(x)$ the Jacobian of $F$, and $f(x)=||F(x)||^2$. Note that $\nabla f(x)=2H(x)^{\intercal}F(x)$.  Hence the last While loop in the  algorithm  terminates after a finite time (see \cite{truong2021}). 

\medskip
{\color{blue}
 \begin{algorithm}[H]
\SetAlgoLined
\KwResult{Find a zero of $F:\mathbb{R}^m\rightarrow \mathbb{R}^m$}
Given: $0<\delta_0,\delta_1,\ldots ,\delta _m$  and $0<\tau $;\\
Define: $H(x):=JF(x)$ and $f(x)=||F(x)||^2$;\\
Define: $\kappa :=\inf _{i\not= j}|\delta _i-\delta _j|/2$;\\
Initialization: $x_0\in \mathbb{R}^m$\;
 \For{$k=0,1,2\ldots$}{ 
    $j=0$\\
  \If{$F(x_k)\neq 0$}{
  \If{$minsp(\nabla ^2f(x_k))>||F(x_k)||^{\tau}$}
   {  ~~~~\While{$minsp(\nabla ^2f(x_k)+\delta _j||F(x_k)||)<\kappa ||F(x_k)||$}{j:=j+1}
   ~~~$A_k:=\nabla ^2f(x_k)+\delta _j||F(x_k)||\times Id$\\
  {\bf else}\\
  ~~~~\While{$minsp(\nabla ^2f(x_k)+\delta _j||F(x_k)||^{\tau})<\kappa ||F(x_k)||^{\tau}$}{j:=j+1}
   ~~~$A_k:=\nabla ^2f(x_k)+\delta _j||F(x_k)||^{\tau}\times Id$}
  }
$v_k:=A_k^{-1}H(x_k)^{\intercal}F(x_k)=pr_{A_k,+}v_k+pr_{A_k,-}v_k$\\
 $w_k:=pr_{A_k,+}v_k-pr_{A_k,-}v_k$\\
$\widehat{w_k}:=w_k/\max \{1,||w_k||\}$\\
$\gamma :=1$\\
 \If{$H(x_k)^{\intercal}F(x_k)\neq 0$}{
   \While{$f(x_k-\gamma \widehat{w_k})-f(x_k)>-\gamma <\widehat{w_k},H(x_k)^{\intercal}F(x_k)>$}{$\gamma =\gamma /2$}}

$x_{k+1}:=x_k-\gamma \widehat{w_k}$
   }
  \caption{New Q-Newton's method Backtracking SE} \label{table:alg0}
\end{algorithm}
}
\medskip

Other variants, as in \cite{truong2021}, can be given. See the item "A general second order algorithm" below. Versions on settings different from Euclidean space can also be given, following the ideas in \cite{truong4, truongnew}. 

A well-known algorithm  for solving systems of equations is Levenberg-Marquart algorithm \cite{wikipedia} (and modifications). We propose the following modification. It is simpler than the above version, thanks to the fact that $H(x)^{\intercal}H(x)$ is symmetric and semi-positive. 

\medskip
{\color{blue}
 \begin{algorithm}[H]
\SetAlgoLined
\KwResult{Find a zero of $F:\mathbb{R}^m\rightarrow \mathbb{R}^m$}
Given: $0<\delta_0,\delta_1$  and $0<\tau $;\\
Define: $H(x):=JF(x)$ and $f(x)=||F(x)||^2$;\\
Define: $\kappa :=|\delta _1-\delta _0|/2$;\\
Initialization: $x_0\in \mathbb{R}^m$\;
 \For{$k=0,1,2\ldots$}{ 
    $j=0$\\
  \If{$F(x_k)\neq 0$}{
  \If{$minsp(H(x_k)^{\intercal}H(x_k))>||F(x_k)||^{\tau}$}
   {$A_k:=H(x_k)^{\intercal}H(x_k)+\delta _0||F(x_k)||\times Id$\\
  {\bf else}\\
   $A_k:=H(x_k)^{\intercal}H(x_k)+\delta _1||F(x_k)||^{\tau}\times Id$}
  }
 $w_k:=A_k^{-1}H(x_k)^{\intercal}F(x_k)$\\
$\widehat{w_k}:=w_k/\max \{1,||w_k||\}$\\
$\gamma :=1$\\
 \If{$H(x_k)^{\intercal}F(x_k)\neq 0$}{
   \While{$f(x_k-\gamma \widehat{w_k})-f(x_k)>-\gamma <\widehat{w_k},H(x_k)^{\intercal}F(x_k)>$}{$\gamma =\gamma /2$}}

$x_{k+1}:=x_k-\gamma \widehat{w_k}$
   }
  \caption{Levenberg-Marquardt Backtracking M} \label{table:alg1}
\end{algorithm}
}
\medskip

We have the following result. Note that by definition, for the $A_k$ in either algorithm, we have $minsp(A_k)\geq ||F(x_k)||^{\tau }$, and near a non-degenerate root of $F(x)$ we have $A_k=H(x_k)^{\intercal}H(x_k)+O(||F(x_k)||)$. 

\begin{theorem}

Let $F:\mathbb{R}^m\rightarrow \mathbb{R}^m$ be a $C^1$ function. Define $H(x)=JF(x)$ and $f(x)=||F(x)||^2$. Let $x_0\in \mathbb{R}^m$ be an initial point, and $\{x_n\}$ the corresponding constructed sequence from New Q-Newton's method Backtracking SE or Levenberg-Marquardt M.  

0) (Descent property) $f(x_{n+1})\leq f(x_n)$ for all n. 

1) If $x_{\infty}$ is a {\bf cluster point} of $\{x_n\}$, then $H(x_{\infty})^{\intercal}F(x_{\infty})=0$. That is, $x_{\infty}$ is a {\bf critical point} of $f$.

2) If $0<\tau <1$: If $f$ has at most countably many critical points, then either $\lim _{n\rightarrow\infty}||x_n||=\infty$ or $\{x_n\}$ converges to a critical point of $f$. Moreover, if $f$ has compact sublevels, then only the second alternative happens. 

3) If $x_n$ converges to $x_{\infty}$ which is a non-degenerate zero  of $F$ (that is, if $H(x)$ is invertible at $x_{\infty}$), then the rate of convergence is quadratic. 

4) If $0<\tau <1$: (Capture theorem) If $x_{\infty}'$ is an isolated zero of $F$, then for initial points $x_0'$ close enough to $x_{\infty}'$, the sequence $\{x_n'\}$  constructed by New Q-Newton's method Backtracking SE will converge to $x_{\infty}'$. 

\label{Theorem1}\end{theorem}

\begin{proof}

Parts 0, 1 and 2 follow exactly as in \cite{truong2021} for New Q-Newton's Backtracking. 

Part 3: this follows because near $x_{\infty}$ then the update rule for either algorithm is different from the usual Newton's method only in $O(||F(x_k)||)$, which will still has quadratic rate of convergence. 

Part 4: this is well known for optimization algorithms having the descent property in part 0.

\end{proof}

Next, we discuss the avoidance of saddle point. Recall that a point $x^*$ is a generalised saddle point of a function $f$ if $\nabla f(x^*)=0$, and moreover $\nabla ^2f(x^*)$ has at least one negative eigenvalue. Note that if $x^*$ is a generalised saddle point of $f(x)=||F(x)||^2$, then $x^*$ cannot be a root of $F(x)=0$. In particular, since $\nabla f(x)/2=H(x)^{\intercal}F(x)$, we have that $H(x^*)$ is singular, and hence $A(x)=H(x)^{\intercal}H(x)+\delta _1||F(x)||^{\tau}$ near $x^*$. 

We consider the Levenberg-Marquardt Backtracking M algorithm first, which is more complicated to deal with. The main reason is that the main term $H(x)^{\intercal}H(x)$ is not the same as $\nabla ^2f(x)$. Note that the dynamics of Levenberg-Marquardt Backtracking M is $x\mapsto G(x)=x-\gamma (x)(H(x)^{\intercal}H(x)+\delta _1||F(x)||^{\tau})^{-1}.\nabla f(x)/2$, and it is easy to compute that $G(x)=(H(x)^{\intercal}H(x)+\delta _1||F(x)||^{\tau})^{-1}.\nabla f(x)/2$ is $C^1$ near $x^*$, and $\nabla G(x^*)=(H(x^*)^{\intercal}H(x^*)+\delta _1||F(x^*)||^{\tau})^{-1}.\nabla ^2f(x^*)/2$. Since  $H(x^*)^{\intercal}H(x^*)+\delta _1||F(x^*)||^{\tau}$ is strictly positive definite near $x^*$, it follows that the term $(H(x^*)^{\intercal}H(x^*)+\delta _1||F(x^*)||^{\tau})^{-1}.\nabla ^2f(x^*)/2$ also has at least one negative eigenvalue. Therefore, if $\gamma (x)$ varies $C^1$ near $x^*$ (for example, if it is a constant), then one has a local Stable-Central manifold for $x^*$ for the dynamics of Levenberg-Marquardt Backtracking M, and then can use the results in \cite{truong2021} to show (when $\delta _0, \delta _1$ are randomly chosen from beginning) that Levenberg-Marquardt Backtracking M can globally avoid $x^*$. 

There is one way to make $\gamma(x)$ to be constant near $x^*$, that is to choose $\gamma (x)$ by a more sophisticated manner. The main idea in \cite{truong4, truongnew}, used for Backtracking line search for gradient descent, is to find a continuous quantity $R(x)$, so that if $\gamma <R(x)$ then Armijo's condition is satisfied. Then, if one chooses $\gamma (x)$ in a sequence $\{\beta ^n: ~n=0,1,2,\ldots \}$, where $0<\beta <1$ is randomly chosen, then one can make sure (in case $x^*$ is an isolated saddle point) that $\gamma (x)$ is constant near $x^*$. In the case at hand, such an $R(x)$ can be bounded from the quantities involving $F(x)$ at $x^*$. The case of non-isolated saddle points can be treated similarly, by using Lindelof's lemma (which is first used in \cite{panageas-piliouras} for the usual Gradient descent method). Therefore, we have the following result. 

\begin{theorem}
If $F:\mathbb{R}^m\rightarrow \mathbb{R}^m$ is $C^2$, and $f=||F||^2$, then by choosing $\delta _0,\delta _1$ randomly, as well as choosing the learning rate $\gamma (x)$ in a sequence $\{\beta ^n: ~n=0,1,2,\ldots \}$ (where $0<\beta <1$ is randomly chosen) by the manner in \cite{truong4, truongnew}, then for a random initial point the sequence $\{x_n\}$ constructed by Levenberg-Marquardt Backtracking M cannot converge to a generalised saddle point. 

\label{Theorem4}\end{theorem}   

However, it is still preferable to show avoidance of saddle points for the original version of Levenberg-Marquardt Backtracking M here, since it is simpler. It is expected that settling this question - at least locally near the saddle point - is the same as settling the question of whether the original version of Armijo's Backtracking line search for Gradient descent can avoid saddle points, and the latter question is still open. (As mentioned, a more sophisticated choice of learning rate for Backtracking line search for Gradient descent can avoid saddle points, see \cite{truong4, truongnew}.) Here, by using the ideas in \cite{truong2021}, where New Q-Newton's method Backtracking is shown to avoid saddle points, we can show that Levenberg-Marquardt Backtracking M can avoid saddle points of a special type, which is described next.

Again, let $f(x)=||F(x)||^2$ and $x^*$ a saddle point of $f$. Then $\nabla ^2f(x)/2 =H(x)^{\intercal}H(x)+\nabla ^2F(x).F(x)$ has at least $1$ negative eigenvalue. The special generalised saddle points we concern are: 

{\bf Strong generalised saddle points:} $x^*$ is a strong generalised saddle point of $f(x)=F(x)^2$ if it is a generalised saddle point of $f$, and moreover $\nabla ^2F(x^*)F(x^*)$ is negative definite.

\begin{theorem}
Let $F:\mathbb{R}^m\rightarrow \mathbb{R}^m$ be $C^2$ and $f(x)=||F(x)||^2$. Assume that $\delta _0,\delta _1$ are chosen randomly. If $x_0$ is a random initial point, and $\{x_n\}$ is the sequence  constructed from Levenberg-Marquardt Backtracking M, then $\{x_n\}$ cannot converge to a strong generalised saddle point $x^*$ of $f$.  

\label{Theorem5}\end{theorem}
\begin{proof}
As mentioned before the statement of Theorem \ref{Theorem4}, it suffices to show that for $x$ close to $x^*$, then the learning rate $\gamma (x)=1$.  Denote, as usual,  $w(x)=(H(x)^{\intercal}H(x)+\delta _1||F(x)||^{\tau})^{-1}.H(x)^{\intercal}.F(x)$. Note that $||w(x)||\sim ||H(x)^{\intercal}.F(x)||$ Then, by Taylor's expansion we have
\begin{eqnarray*}
f(x-w(x))-f(x)&=&-2<w(x),H(x)^{\intercal}.F(x)>\\
&&+<[H(x)^{\intercal}H(x)+\nabla ^2F(x).F(x)].w(x),w(x)>+o(||w(x)||^2)
\end{eqnarray*}
 
Since $\nabla ^2F(x).F(x)$ is negative definite when $x$ is close to $x^*$, and since $\delta _1>0$, we obtain 
\begin{eqnarray*}
<[H(x)^{\intercal}H(x)+\nabla ^2F(x).F(x)].w(x),w(x)>&\leq&<[H(x)^{\intercal}H(x)+\delta _1||F(x)||^{\tau }]w(x), w(x)>\\
&=&<H(x)^{\intercal}.F(x),w(x)>.  
\end{eqnarray*}

Hence, for $x$ near $x^*$, we have as wanted: $f(x-w(x))-f(x)\leq - 1/3<w(x),H(x)^{\intercal}.F(x)>$. 
\end{proof}

Note that when $m=1$, then a saddle point of $f$ is also a strong generalised saddle point. On the other hand, for $m=2$, the two notions are different, already for the case $f=|g(z)|^2$ where $g$ is a univariate holomorphic function (see \cite{truong-etal}). 

Concerning avoidance of saddle points, currently New Q-Newton's method Backtracking SE has better theoretical guarantees. The proof is similar to the proof of Theorem \ref{Theorem5}.

\begin{theorem}
Let $F:\mathbb{R}^m\rightarrow \mathbb{R}^m$ be $C^2$ and $f(x)=||F(x)||^2$. Assume that $\delta _0,\delta _1$ are chosen randomly. If $x_0$ is a random initial point, and $\{x_n\}$ is the sequence  constructed from New Q-Newton's method Backtracking SE or Levenberg-Marquardt Backtracking M, then $\{x_n\}$ cannot converge to a generalised saddle point $x^*$ of $f$.  

\label{Theorem6}\end{theorem}

Recall that a function $f$ satisfies Lojasiewicz gradient inequality at a point $x^*$ if there is a small neighbourhood $U$ of $x^*$, a constant $0<\mu <1$ and a constant $C>0$ so that for all $x,y\in U$ we have
\begin{eqnarray*}
|f(x)-f(y)|^{\mu}\leq C||\nabla f(x)||. 
\end{eqnarray*}

We then can define:  

{\bf Definition (Lojasiewicz exponent):} Assume that $f:\mathbb{R}^m\rightarrow \mathbb{R}$ has the Lojasiewicz gradient inequality near its critical points. Then at each critical point $x^*$ of $f$, we define

$\mu (x^*):=\inf \{\mu : $ there is an open neighbourhood $U$ of $x^*$ and a constant $C>0$ so that for all $x,y\in U$ we have $|f(x)-f(y)|^\mu \leq C||\nabla f(x)||\}$.  

\begin{theorem} Assume that $F:\mathbb{R}^m\rightarrow \mathbb{R}^m$ is $C^1$ so that $f=||F||^2$ satisfies the Lojasiewicz gradient inequality. Let $\{x_n\}$ be a sequence constructed by New Q-Newton's Backtracking SE. Assume also that $0<\tau < 1$. 

1) Assume that for all critical points $x^*$ of $f$, we have $ \mu (x^*)\times (1+\tau ) <1$. Then either $\lim _{n\rightarrow\infty}||x_n||=\infty$ or $\{x_n\}$ converges to a critical point of $f$. 

2) If $F$ is a polynomial map, then the condition in part 1) is satisfied, provided $\tau >0$ is small enough. 
\label{Theorem3}\end{theorem}  
\begin{proof}
The proof is similar to what given in \cite{truong2021} for New Q-Newton's method Backtracking, where in part 2 an effective upper bound \cite{acunto-kurdyka} for the Lojasiewicz exponent of a polynomial map (in terms of degree of the map and the dimension)  is used. 
\end{proof}

{\bf A general second order algorithm:} Here we present a general scheme for to have good theoretical guarantee second order algorithms, combining ideas in this paper and \cite{truong2021}: 

Assume that one wants to optimise a $C^2$ cost function $f:\mathbb{R}^m\rightarrow \mathbb{R}$ of the special form $f(x)=||F(x)||^2$. 

One fixes $0<\delta _0,\delta _1,\ldots ,\delta _m$ (randomly chosen) real numbers. One fixes also numbers $0<\tau <1$ and $q\geq 1$. 

At each point $x\in \mathbb{R}^m$, assume that one is given a symmetric matrix $B(x)$ (not necessarily semi-positive) and an orthonormal basis $e_1(x),\ldots ,e_m(x)$ of $\mathbb{R}^m$.  

One defines $2\kappa :=\min _{i\not= j}|\delta _i-\delta _j|$. 

If $minsp(B(x))>||F(x)||^{\tau}$, then one chooses $\delta _j$ to be the first element in $\{\delta _0,\ldots ,\delta _m\}$ so that $A(x)=B(x)+\delta _j||F(x)||\times Id$ satisfies $minsp(A(x))\geq \kappa ||F(x)||$. Otherwise, one  chooses $\delta _j$ to be the first element in $\{\delta _0,\ldots ,\delta _m\}$ so that $A(x)=B(x)+\delta _j||F(x)||^{\tau}\times Id$ satisfies $minsp(A(x))\geq \kappa ||F(x)||^{\tau}$.  

One chooses the search direction: 
\begin{eqnarray*}
w(x):=\sum _{i=1}^m\frac{<\nabla f(x),e_i(x)>}{||A(x).e_i(x)||_{q,e_1(x),\ldots ,e_m(x)}}e_i(x),  
\end{eqnarray*}
where $||A(x).e_i(x)||_{q,e_1(x),\ldots ,e_m(x)}:=[\sum _{j=1}^m|<A(x).e_i(x),e_j(x)>|^q]^{1/q}$. 

One defines $\gamma (x)$ by Armijo's Backtracking line search w.r.t. the cost function $f(x)$ and the search vector $\widehat{w}(x)=w(x)/\min \{||w(x)||,1\}$. 

Then one has the update rule $x\mapsto x-\gamma (x)\widehat{w}(x)$. 

As discussed above, the following result helps to show that the above general scheme avoids saddle points. 

\begin{theorem} Let $f(x)=||F(x)||^2$ be as above. Let $w(x)$ be chosen as in the item "A general second order algorithm" above. Assume that $q\geq 1$ is so that its H\"older's conjugate $p$ (i.e. the number $p\geq 1$ so that $p^{-1}+q^{-1}=1$, the value $+\infty$ is allowed) satisfies $m^{1/p}<4/3$. Assume moreover that for $x$ close enough to a generalised saddle point $x^*$ of $f$, we have $B(x)-\nabla ^2f(x)+O(||\nabla f(x)||)$ is positive definite. Then, for all $x$ close enough to $x^*$ we have $\gamma (x)=1$. 

\label{Theorem7}\end{theorem}
\begin{proof}
Since $||F(x)||$ is bounded away from 0 near a generalised saddle point $x^*$ of $f$, we have that $||w(x)||\sim ||\nabla f(x)||$ near $x^*$. Also, $<w(x),\nabla f(x)>\sim ||w(x)||^2$.   By Taylor's expansion we have
\begin{equation}
f(x-w(x))-f(x)=-<w(x),\nabla f(x)>+\frac{1}{2}<\nabla ^2f(x)w(x),w(x)>+O(||w(x)||^3). 
\label{Equation11}\end{equation}
Denote $a_i=\nabla f(x),e_i(x)$, $b_{i,j}=<A(x)e_i(x),e_j(x)>$, and $B_i=(\sum _{j=1}^m|b_{i,j}|^q)^{1/q}$. We have
\begin{eqnarray*}
<w(x),\nabla f(x)>=\sum _{i=1}^ma_i^2/B_i. 
\end{eqnarray*}

 Now we bound from above the second summand in the RHS in Equation (\ref{Equation11}), note that all $\delta _j>0$: 
\begin{eqnarray*}
<\nabla ^2f(x)w(x),w(x)>&\leq& <[B(x)+O(||w(x)||].w(x),w(x)>\\
&\leq& <A(x)w(x),w(x)>+O(||w(x)||^3)\\
&=&<\sum _{i=1}^ma_iA(x)e_i(x)/B_i,\sum _{j=1}^ma_je_j(x)/B_j>\\
&=&\sum _{i=1}^m\sum _{j=1}^ma_ia_jb_{i,j}/(B_iB_j)
\end{eqnarray*}
Now, by Cauchy-Schwartz inequality, we have
\begin{eqnarray*}
|a_ia_jb_{i,j}|/(B_iB_j) \leq \frac{a_i^2|b_{i,j}|}{2B_i^2}+\frac{a_j^2|b_{i,j}|}{2B_j^2}
\end{eqnarray*}
By H\"older's inequality, we obtain for all $i$:
\begin{eqnarray*}
\sum _{j=1}^m|b_{i,j}|\leq m^{1/p}B_i
\end{eqnarray*}
This implies, from the assumption on $p$, that
\begin{eqnarray*}
f(x-w(x))-f(x)<-\frac{1}{3}<w(x),\nabla f(x)>
\end{eqnarray*}
for all $x$ close enough to $x^*$. Thus we can choose $\gamma (x)=1$ for all such points. 

\end{proof}

\begin{remark}
0) While $0<\tau <1$ is needed in theoretical proofs, experiments show that the choice of $\tau$ (even of values $\geq 1$) does not really affect the performance. 

In the case where the inverse of $H(x)^{\intercal}H(x)+\delta ||F(x)||^{\tau}$ is expensive to compute, one can use other variants in the item "A general second order algorithm" which does not require computing the inverse matrix.

1) Levenberg-Marquardt algorithm \cite{wikipedia} is extensively studied in the literature. 

- The corresponding operator in  Levenberg-Marquardt algorithm  is $H(x)^{\intercal}H(x)+\epsilon $. In most of the work,  the associated dynamical system is $x\mapsto x-(H(x)^{\intercal}H(x)+\epsilon )^{-1}.H(x)^{\intercal}.F(x)$ (without a learning rate $\gamma (x)$).

- The choice of $\epsilon = ||F(x)||^{\tau}$ has been used extensively in the literature, starting with \cite{yamashita-fukushima, fan-yuan}, who showed that if one chooses $\tau \in [1,2]$ (the choice of $\epsilon =||F(x)||$ was previously suggested in \cite{kelley}) and if $x^*$ is a solution of $F(x)=0$ satisfying a certain "error bound" condition (which is weaker than requiring that $H(x^*)$ is invertible), and $x_0$ is close to $x^*$, then the sequence $\{x_n\}$ converges to $x^*$ with quadratic rate of convergence. "Error bound" is indeed similar to Lojasiewicz inequality, and recently there are works which prove local convergence near $x^*$ a root of $F(x)$ having the Lojasiewicz (gradient) inequality see \cite{ahookhosh-etal}. However, there was no study of convergence issues near critical points of $f(x)=||F(x)||^2$ which are roots of $F(x)$, or how one may avoid them. 

-  There are also works where $\epsilon$ is an interpolation between $||F(x)||^{\tau}$ and $||\nabla f(x)||^{\tau}$, see e.g. \cite{ahookhosh-etal}. However, in terms of avoidance of saddle points, it seems from the results we obtain in this paper that the term  $||F(x)||^{\tau}$ is better than the term $||\nabla f(x)||^{\tau}$.      

 - Armijo's Backtracking line search, to choose learning rate $\gamma (x)$, has been also used in the literature, see e.g. \cite{fan-yuan, bao-etal}, the latter paper also considers the inexact setting. However, the Backtracking line search used in those papers is not pure one as in our paper.  Fix a number $0<\eta <1$. Depending on whether the condition $||F(x_k-w_k)||\leq \eta ||F(x_k)||$, Armijo's condition will not be checked (in which case $\gamma (x)=1$) or will be checked. Indeed, this Backtracking scheme can be incorporated into the algorithms proposed in the current paper to obtain better local rate of convergence. 
 
 - In all of these works, we are not aware of study of global convergence to a critical point of $f(x)$ but not a root of $F(x)$. In particular, the question of avoidance of saddle points is not addressed in these papers.   
 
 A detailed comparison between New Q-Newton's method (Backtracking) and some well known modifications of Newton's method is given in \cite{truong-etal}.   

- Some new contribution in the algorithm Levenberg-Marquardt M proposed in this paper are: 

First, we normalise the search direction $w(x)$ to $w(x)/\max \{1,||w(x)||\}$, which helps to prove global convergence in the case the function $f$ does not have compact sublevels. Second, we use the ideas in \cite{truong-etal, truongnew, truong2021} to establish results on avoidance of saddle points.

2) When $F$ is a polynomial map, the question of solving zeros to $F$ is important in Algebraic Geometry. Purely theoretical tool for solving this  (over an algebraically closed field) is Groebner's basis, but it can be very slow in practice for large systems.  Question 17 in the famous list of open questions by Smale \cite{smale} is to find an algorithm to find such a solution, over $\mathbb{C}$, quickly (quadratic rate of convergence), in the generic situation. This question has been solved about some years ago \cite{beltran-pardo1, beltran-pardo2, cucker-burgisser, lairez}, using the homotopy continuation method and the theory in \cite{shub-smale}. The software Bertini provides numerical computations (for non-generic situations as well) based also on the homotopy continuation method. Currently, homotopy continuation method only applies in $\mathbb{C}$, while usually only solutions over $\mathbb{R}$ have physical meaning. 

There are routines for solving (both symbolically and numerically) in many computational softwares such as Mathematica, Mapple and MatLab. The special case $F(x)=$ a linear map is extensively studied in Numerical Linear Algebra. 
 
3) It could happen that a cluster point $x_{\infty}$ is a zero of $H(x)^{\intercal}F(x)$ but not of $F(x)$. In \cite{mehta-etal}, in the case where $F$ is a polynomial map, a numerical method based again on homotopy continuation method has been proposed with the aim to avoid this. 

Here we propose another approach, which apply to more general systems. We can assume, in a generic situation, that $0$ is a regular value of $F$. This means that for all zero $x^*$ of $F(x)=0$, the Jacobian $JF(x^*)$ is invertible. We can choose a small number $\epsilon >0$, and when choosing the learning rate $\gamma$ by Backtracking line search, also make sure that $|det (JF(x_{k+1}))|>\epsilon $. By letting $\epsilon \rightarrow 0$, we can gradually find all roots of $F$. When $F$ is a polynomial system, in such "generic situations" we have only a finite number of zeros to the system, and hence there is a positive $\epsilon >0$ so that for all $x^*$ with $F(x^*)=0$ then $|det (JF(x^*))|>\epsilon$. Hence, an alternative solution to (the real variable version, which usually has more physical meaning) Smale's 17th problem, and with an easy implementation, could be the following route: first, we establish an effective bound for such an $\epsilon$; and second, one find a way to choose an initial $x_0$ for which the sequence $\{x_n\}$ constructed by New Q-Newton's method Backtracking SE does not diverge to the boundary of the set $|det (JF(x^*))|>\epsilon$ (alternatively, one can modify New Q-Newton's method Backtracking SE a bit to make sure about this).   

4) New Q-Newton's method Backtracking SE can be extended to under/overdetermined systems, that is zeros of $G:\mathbb{R}^m\rightarrow \mathbb{R}^{m'}$, where $m\not= m'$.  We just need to note that $H $ is then an $m'\times m$ matrix, and $H^{\intercal}$ is an $m\times m'$ matrix. Another approach was first used in \cite{graves}, by using kind of pseudo-inverses, and has been extensively studied and extended. In the case where $F$ is a polynomial map, and over the field $\mathbb{C}$, another idea, based on the fact that any variety in $\mathbb{C}^m$ can be set-theoretically defined by at most $m$ equations \cite{storch, eisenbud-evans}, can be found in \cite{mehta-etal}.  

5) Finding isolated intersection points are also useful in various situations. For example, in dynamical systems, it is interesting to find hyperbolic periodic points, which supposedly play an important role in understanding the equilibrium measure (if any) of the given map. These are isolated intersection points between the diagonal and the graph of the iterates of the map. 

\end{remark}

{\bf Implementation and Experiments:} Will be updated at the GitHub link \cite{tuyenGitHub}. 

{\bf Acknowledgments.} The author would like to thank Jonathan Hauenstein for helping with some relevant questions. The author is partially supported by Young Research Talents grant 300814 from Research Council of Norway.

\end{document}